\documentclass[12pt]{amsart}
\usepackage[mathscr]{eucal}
\usepackage{amssymb}
\usepackage{latexsym}
\usepackage{amsthm}
\usepackage{amsmath}
\usepackage{graphicx}

\theoremstyle{plain}

\newtheorem{thm}{Theorem}[section]

\newtheorem{lem}{Lemma}[section]

\theoremstyle{definition}

\newtheorem{dfn}{Definition}[section]
\newtheorem{prf}{Proof}[section]

\title{On additive property of finitely additive measures}
\author{Ryoichi Kunisada}

\address{Faculty of Education and Integrated Arts and Science, Waseda University, Shinjuku-ku, Tokyo 169-8050, Japan}

\email{rkunisada@aoni.waseda.jp}

\thanks{This paper is a part of the outcome of research  performed under a Waseda University Grant for Special Research Projects (Project number: 2018K-152).}

\date{}

\begin{document}
\maketitle

\begin{abstract}
By the additive property, we mean a condition under which $L^p$ spaces over finitely additive measures are complete. 
Basile and Rao gives a necessary and sufficient condition that a finite sum of finitely additive measures has the additive property. We generalize this result to the case of a countable sum of finitely additive meaures. An application of this result to density measures are also presented.
\end{abstract}

\bigskip

\section{Introduction}
Let $X$ be a set and $\mathscr{F}$ be a field of subsets of $X$. A function $\mu$ on $\mathscr{F}$ is called a 
{\it finitely additive measure} or a {\it charge} if the following conditions are satisfied.
\begin{enumerate}
\item $\mu(\emptyset) = 0$,
\item $\mu(A \cup B) = \mu(A) + \mu(B)$ if $A, B \in \mathscr{F}, A \cap B = \emptyset$.
\end{enumerate}

The triple $(X, \mathscr{F}, \mu)$ is called a {\it finitely additive measure space} or {\it charge space}. In what follows, we use the term ‘charge’ exclusively. Obviously, charges are a generalization of measures obtained by replacing countable additivity of measures with finite additivity. The theory of charges was developed systematically in [2]. In this book, various notions and results in measure theory are generalized and transferred to charge spaces.
Among these, the notion of $L^p$ spaces over charges is of particular importance for applications of charge theory. Recall that one of the important conclusions of measure theory is the completeness of $L^p$ spaces over measures. Unfortunately however, $L^p$ spaces over charges are not complete in general. The condition under which $L^p$ spaces over charges are complete is given by a certain additivity property which is intermediate between countable additivity and finite additivity (see [1]): Let $(X, \mathscr{F}, \mu)$ be a charge space. We say that $\mu$ has the {\it additive property} if for any $\varepsilon > 0$ and increasing sequence $A_1 \subseteq A_2 \cdots A_n \subseteq \cdots$ in $\mathscr{F}$, there exists a set $B \in \mathscr{F}$ such that
\begin{enumerate}
\item $\mu(B) \le \lim_{i \to \infty}  \mu(A_i) + \varepsilon$,
\item $\mu(A_i \setminus B) = 0$ \quad for every \ $i = 1,2, \ldots$. 
\end{enumerate}

The study of the additive property for concrete examples of charges were developed, for instance, in [3], [4], [5].

In this paper, we deal with the additive property of sums of charges. Note that if charges $\mu$ and $\nu$ on $(X, \mathscr{F})$ have the additive property, $\mu+\nu$ does not necessarily have the additive property. In [ ], the necessary and sufficient condition that finite sums of charges have the additive property has been obtained.  One of the main aims of this paper is to generalize the result to the case of countable sums of charges. Further, we give an application of this result to the additive property of density measures, which gives a simpler proof of the main result of [4]. This can be seen as another main result of this paper.

The paper is organized as follows. In Section 2, we introduce some notions and results which will be used throughout the paper. In particular, the notions of singularity and strong singularity of charges and  a Borel measure on the stone space of an algebra of sets play an important role in studying the additive proeperty.

In Section 3, we chronicle equivalent conditions to the additive property, which illustrate the importance of the additive property in the theory of charges. In fact, these theorems asserts that some of the main theorems in meaure theory, including the completeness of $L^p$ spaces and the Radon-Nikodym theorem, are also valid for charges having the additive proeperty.

Section 4 deals with one of the main results of this paper, i.e., we prove a necessary and sufficient condition that a countable sum of charges has the additive property. This is done by using one of the equivalent formulations of the additive property introduced in Section 3.

Section 5 is devoted to an application of the result of Section 4.

\section{Preliminaries}
In what follows, if not otherwise stated, charges will always be nonnegative and bounded, i.e., any given charge space $(X, \mathscr{F}, \mu)$, we have $\mu(A) \ge 0$ for every $A \in \mathscr{F}$ and $\mu(X) < \infty$.

First, for a charge space $(X, \mathscr{F}, \mu)$ we introduce an extension of $\mu$ to a regular Borel measure of the stone space of $\mathscr{F}$. This method plays an important role for formulating various notions concerning charges. 
Regarding $\mathscr{F}$ as a Boolean algebra, by the Stone representation theorem, there exists a totally disconnected compact space $F$ and a natural Boolean isomorphism $\phi : \mathscr{F} \rightarrow \mathscr{C}$, where $\mathscr{C}$ is the algebra of the clopen subsets of $F$.

Now define a charge $\hat{\mu}$ on $\mathscr{C}$ by $\hat{\mu}(\phi(A)) = \mu(A)$ and we get a charge 
space $(F, \mathscr{C}, \hat{\mu})$.
Since any infinite union of clopen subsets can not be a clopen subset, $\hat{\mu}$ is countably additive on $\mathscr{C}$ and thus by the E. Hopf extension theorem, we can extend it to a countable additive measure on the $\sigma$-algebra generated by $\mathscr{C}$, i.e., a Baire measure on $F$. This can also be extended to a regular Borel measure on $F$ in a unique way. We still denote it by $\hat{\mu}$ and thus we obtain a measure space $(F, \mathcal{B}(F), \hat{\mu})$, where $\mathcal{B}(F)$ is the Borel sets of $F$. We denote  by supp $\mu$ the support of $\hat{\mu}$ in $F$. 

\medskip
Following [2], we consider the notions of absolute continuity and singularity for charges. 

\begin{dfn}
Let $\mu$ and $\nu$ be charges on $(X, \mathscr{F})$.
\begin{enumerate}
\item We say that $\nu$ is {\it absolutely continuous} with respect to $\mu$ if for any $\varepsilon > 0$, there exists $\delta > 0$ such that $\nu(A) < \varepsilon$ whenever $\mu(A) < \delta$, where $A \in \mathscr{F}$. In this case, we write $\nu \ll \mu$.
\item We say that $\mu$ and $\nu$ are {\it singular} if for every $\varepsilon > 0$, there exists a set $D \in \mathscr{F}$ such that $\mu(D) < \varepsilon$ and $\nu(D^c) < \varepsilon$.
In this case, we write $\mu \perp \nu$.
\end{enumerate}
\end{dfn}

Next we define the notions of weakly absolute continuity and strong singularity.

\begin{dfn} 
Let $\mu$ and $\nu$ be charges on $(X, \mathscr{F})$.
\begin{enumerate}
\item We say that $\nu$ is {\it weakly absolutely continuous} with respect to $\mu$ if $\nu(A) = 0$ whenever $\mu(A) = 0$, where $A \in \mathscr{F}$. In this case, we write $\nu \prec \mu$.
\item We say that $\mu$ and $\nu$ are {\it strongly singular} if there exists a set $D \in \mathscr{F}$ such that $\mu(D) = 0$ and $\nu(D^c) = 0$. In this case, we write $\mu \ \rotatebox{90}{$\vDash$} \ \nu$.
\end{enumerate}
\end{dfn}

Obviously, these are ordinary notions of absolute continuity and singularity for measures in case that $\mathscr{F}$ is a $\sigma$-algebra and $\mu$ and $\nu$ are measures on it. In fact, as the following theorems show, absolute continuity and weakly absolute continuity, and singularity and strong singularity coincide respectively in this case. See [2] for the proofs.

\begin{thm} 
Let $\mu$ and $\nu$ be measures on $(X, \mathscr{F})$. Then we have the following results.
\begin{enumerate}
\item If $\nu$ is a bounded measure, then $\nu \ll \mu$ if and only if $\nu \prec \mu$.
\item $\nu \perp \mu$ if and only if  $\mu \ \rotatebox{90}{$\vDash$} \ \nu$.
\end{enumerate}
\end{thm}

Now we give formulations of these notions by using the extended measure space. The following is essentially due to [1].

\begin{thm} 
Let $\mu$ and $\nu$ be charges on $(X, \mathscr{F})$ and $\hat{\mu}$ and $\hat{\nu}$ be the extended measures on $(F, \mathcal{B}(F))$, respectively. Then the following statements hold:
\begin{enumerate}
\item $\nu \ll \mu$ if and only if $\hat{\nu} \prec \hat{\mu}$.
\item $\nu \prec \mu$ if and only if supp $\nu$ $\subseteq$ supp $\mu$.
\item $\nu \perp \mu$ if and only if $\hat{\nu} \ \rotatebox{90}{$\vDash$} \ \hat{\mu}$.
\item $\mu \ \rotatebox{90}{$\vDash$} \ \nu$ if and only if supp $\nu$ $\cap$ supp $\mu$ = $\emptyset$.
\end{enumerate}
\end{thm}

\begin{prf}
(1) Let us assume that $\nu \ll \mu$. Given $\varepsilon > 0$, take $\delta > 0$ as above.
Let $E \in \mathcal{B}(F)$ with $\hat{\mu}(E) = 0$. Since $\hat{\mu}$ is regular, for any $\delta > \delta^{\prime} > 0$, one can choose $A \in \mathscr{F}$ such that $\hat{\mu}(\phi(A) \triangle E) < \delta^{\prime}$ and $\hat{\nu}(\phi(A) \triangle E) < \delta^{\prime}$. Thus we have $\mu(A) = \hat{\mu}(\phi(A)) \le \hat{\mu}(\phi(A) \triangle E) + \hat{\mu}(E) = \hat{\mu}(\phi(A) \triangle E) < \delta^{\prime}$. On the other hand, $\hat{\nu}(E) \le \hat{\nu}(\phi(A) \triangle E) + \hat{\nu}(\phi(A)) \le \delta^{\prime} + \varepsilon$. Since $\delta^{\prime}$ and $\varepsilon$ can be arbitrary small, we have $\hat{\nu}(E) = 0$. Thus $\hat{\nu} \prec \hat{\mu}$.

Conversely, suppose that $\hat{\nu} \prec \hat{\mu}$. Since $\mathcal{B}(F)$ is a $\sigma$-algebra and $\hat{\mu}$ and $\hat{\nu}$ are measures, $\hat{\nu} \prec \hat{\mu}$ if and only if $\hat{\nu} \ll \hat{\mu}$ by Theorem 2.1. Thus through $\phi^{-1}$ we have $\nu \ll \mu$.

(2) This is obvious.

(3) Let us assume that $\nu \perp \mu$. Let $\varepsilon_1 > 0$ and $D_{\varepsilon_1,i} \in \mathscr{F}$, $i \ge 1$ be such that $\mu(D_{\varepsilon_1, i}) < \frac{\varepsilon_1}{2^i}$ and $\nu(D_{\varepsilon_1, i}^c) < \frac{\varepsilon_1}{2^i}$. Notice that $\nu(D_{\varepsilon_1, i}) > 1-\frac{\varepsilon_1}{2^i}$ for every $i \ge 1$. Hence we have $\hat{\nu}(\cup_{i \ge 1} \phi(D_{\varepsilon_1, i})) = 1$. On the other hand, $\hat{\mu}(\cup_{i \ge 1} \phi(D_{\varepsilon_1, i})) \le \sum_{i=1}^{\infty} \frac{\varepsilon_1}{2^i} = \varepsilon_1$. Now we choose a decreasing sequence $\{\varepsilon_j\}_{j \ge 1}$ of positive numbers such that $\lim_{j \to \infty} \varepsilon_j = 0$. Then we put $E = \cap_{j \ge 1} \cup_{i \ge 1} \phi(D_{\varepsilon_j, i})$ and we get $\hat{\nu}(E^c) = 0$ and $\hat{\mu}(E) = 0$, which means that $\mu \ \rotatebox{90}{$\vDash$} \ \nu$.

Now suppose that $\mu \ \rotatebox{90}{$\vDash$} \ \nu$. Then there exists a set $D$ in $\mathcal{B}(F)$ such that $\hat{\mu}(D) = 0$ and $\hat{\nu}(D^c) = 0$. By the regularity of $\hat{\mu}$ and $\hat{\nu}$, there exists a set $C$ in $\mathscr{F}$ such that $\hat{\mu}(\phi(C) \triangle D) < \varepsilon$ and $\hat{\nu}(\phi(C) \triangle D) < \varepsilon$. Thus we have $\mu(C) = \hat{\mu}(\phi(C)) \le \hat{\mu}(\phi(C) \triangle D) + \hat{\mu}(D) < \varepsilon$ and $\nu(C^c) = \hat{\nu}(\phi(C^c)) \le \hat{\nu}(\phi(C^c) \triangle D^c) + \hat{\nu}(D^c) = \hat{\nu}(\phi(C) \triangle D) + \hat{\nu}(D^c) < \varepsilon$. This shows that $\nu \perp \mu$.

(4) This is obvious.
\end{prf}

Concerning these notions, we give a generalization of the Lebesgue decomposition theorem to charges (Refer to [2] for details).

\begin{thm}
For given charges $\mu$ and $\nu$ on $(X, \mathscr{F})$, there exists charges $\nu_1$ and $\nu_2$ on $(X, \mathscr{F})$ such that
\begin{enumerate}
\item $\nu = \nu_1 + \nu_2$.
\item $\nu_1 \ll \mu$.
\item $\nu_2 \perp \mu$.
\end{enumerate}
Furthermore, a decomposition of $\nu$ satisfying (2) and (3) is unique.
\end{thm}


\section{Equivalent conditions to additive property}
In this section, we introduce some equivalent assertions to the additive property. As we have mentioned in Section 1, one can generalize some of the main theorems in measure theory to charges having the additive property. In fact, conversely, the validity of these theorems are also sufficient conditions for charges to have the additive property. We begin with the completeness of $L^p$ spaces over charges, which is the original motive of introducing the notion of the additive property. See [1] for the proofs of the following results.

\begin{thm}
For a charge $\mu$ (not necessarily bounded) on $(X, \mathscr{F})$, $\mu$ has the additive property if and only if $L^p(\mu)$ is complete.
\end{thm}

The next result is a generalization of the Radon-Nikodym theorem to charges.

\begin{thm}
For a charge $\mu$ on $(X, \mathscr{F})$, $\mu$ has the additive property if and only if 
for every charge $\nu$ (not necessarily nonnegative) on $(X, \mathscr{F})$ with $\nu \ll \mu$, there exists some $f \in L^1(\mu)$ such that $\nu(A) = \int_A f d\mu$ holds for every $A \in \mathscr{F}$.
\end{thm}

The Hahn decomposition theorem can be generalized to charges as follows.

\begin{thm}
For a charge $\mu$ on $(X, \mathscr{F})$ where $\mathscr{F}$ is a $\sigma$-algebra, $\mu$ has the additive property if and only if for every charge $\nu$ (not necessarily nonnegative) on $(X, \mathscr{F})$ with $\nu \ll \mu$, there exists some $A \in \mathscr{F}$ satisfying the following property; for each $B \in \mathscr{F}$ with $B \subseteq A$, $\nu(B) \ge 0$ holds, and each $B \in \mathscr{F}$ with $B \subseteq A^c$, $\nu(B) \le 0$ holds.
\end{thm}

Finally, we give the following formulation of the additive property using extended measure spaces, which plays an important role in proving the main theorem in the following section. 
As we have seen in Section 2, we extend a charge space $(X, \mathscr{F}, \mu)$ to a measure space $(F, \mathcal{B}(F), \hat{\mu})$. 

\begin{thm}
A charge $\mu$ on $(X, \mathscr{F})$ has the additive property if and only if $\hat{\mu}(U) = \hat{\mu}(\overline{U})$ for every open sets $U$ of supp $\mu$, where $\overline{U}$ denotes the closure of $U$ in supp $\mu$.
\end{thm}

\section{Main result}
 In this section, we consider a necessary and sufficient condition that charges which are expressed by sums of charges have the additive property. Generally, as shown in [1], if charges $\mu$ and $\nu$ on $(X, \mathscr{F})$ have the additive property, the sum $\mu + \nu$ need not have the additive property.
First we have the following result.

\begin{thm}
Let $\mu$, $\nu$ be charges on $(X, \mathscr{F})$ such that $\nu \ll \mu$. If $\mu$ has the additive property, then $\nu$ has the additive property. 
\end{thm}

From this result together with the Lebesgue decomposition theorem, it is sufficient to consider the condition for pairs of charges $\mu, \nu$ which are mutually singular. It is given by the following, in which a slightly general result of the additive property of finite sums of charges are treated. 

\begin{thm}
Let $\mu_1, \mu_2, \ldots, \mu_n$ be mutually singular charges one another on $(X, \mathscr{F})$ where $\mathscr{F}$ is a $\sigma$-algebra. Then $\mu_1 + \mu_2 + \ldots + \mu_n$ has the additive property if and only if every $\mu_i$, $1 \le i \le n$, has the additive property and they are mutually strongly singular.
\end{thm}

Note that a more general result was proved in [1], namely, the case in which $\mathscr{F}$ is not necessarily a $\sigma$-algebra. 
It is obtained by replacing in the above assertion strong singularity with {\it separability}, in between singularity and strong singularity. The notion of separability is rather complicated than strong singularity and thus we will confine ourselves to the case of $\sigma$-algebra.

Now we consider an extension of Theorem 4.2 to the case of countable sums of charges, which is one of the main results of this paper.

\begin{thm}
Let $\{\mu_i\}_{i \ge1}$ be a countable family of charges on $(X, \mathscr{F})$ where $\mathscr{F}$ is a $\sigma$-algebra such that they are mutually singular and $\mu = \sum_{i \ge 1} \mu_i$ exists. Let $S_i$ be the support of $\mu_i$ and $S$ be the support of $\mu$. Then $\mu$ has the additive property if and only if each $\mu_i$ has the additive property and they are mutually strongly singular and
\[\left(\limsup_i S_i\right) \ \bigcap \ \bigcup_{i \ge 1} S_i = \emptyset\]
holds, where $\limsup_i S_i = \cap_{i \ge 1} \overline{\cup_{j \ge i} S_j}$.
\end{thm}

\begin{prf}
(Sufficiency) We prove the condition in Theorem 3.4. By the definition of $\mu$, it holds that $\hat{\mu} = \sum_{i \ge 1} \hat{\mu}_i$ and thus $\hat{\mu}$ is on $\cup_{i \ge 1} S_i$. Also since $\hat{\mu}_i$ are mutually strongly singular, $\hat{\mu}(A) = \hat{\mu}(A \cap \cup_{i \ge 1} S_i) = \sum_{i \ge 1} \hat{\mu}(A \cap S_i)$ holds. 

In what follows, for any Borel set $X$ of $F$ we denote the closure of $B \subseteq X$ in $X$ by $\overline{B}^X$. In the case of $X = F$, we omit the superscript. Note that the closure of $B \subseteq X$ in X is $X \cap \overline{B}$.
Take any $A \in \mathcal{B}(F)$. By the assumption, note that $S = \cup_{i \ge 1} S_i \cap \limsup_i S_i$ and $\limsup_i S_i \cap \cup_{i \ge 1} S_i = \emptyset$. Since each $\mu_i$ has the additive property, from Theorem 3.4 and the fact that charges $\mu_i$ are mutually strongly singular, we have $\hat{\mu}(\overline{A \cap S_i}) = \hat{\mu}_i(\overline{A \cap S_i}) = \hat{\mu}_i(\overline{A \cap S_i}^{S_i}) = \hat{\mu}_i(A \cap S_i) = \hat{\mu}(A \cap S_i)$.

On the other hand, it holds that $\overline{A \cap S} = \overline{A \cap (\cup_{i \ge1} S_i \cup \limsup_i S_i)} = \overline{\cup_{i \ge 1} (A \cap S_i)} \cup \overline{A \cap \limsup_i S_i} = \cup_{i \ge 1} \overline{A \cap S_i} \cup \overline{A \cap \limsup_i S_i}$. Together with the fact that $\hat{\mu}(\limsup_i S_i) = 0$, we have
\begin{align}
\hat{\mu}(\overline{A \cap S}^S) &= \hat{\mu}(\overline{A \cap S}) = \hat{\mu}(\cup_{i \ge 1} \overline{A \cap S_i}) + \hat{\mu}(\overline{A \cap \limsup_i S_i}) \notag \\
&= \sum_{i \ge 1} \hat{\mu}(\overline{A \cap S_i}) = \sum_{i \ge 1} \hat{\mu}(A \cap S_i) = \hat{\mu}({A \cap S}). \notag 
\end{align}
Hence by Theorem 3.4 we see that $\mu$ has the additive property. 

(Necessity) Suppose that $\{\mu_i\}_{i \ge 1}$ are singular and $\mu = \sum_{i \ge 1} \mu_i$ has the additive property. Let $\mu_n$ and $\mu_m$ be any pair of distinct charges. Put $\mu^{\prime} = \sum_{i \ge 1, i \not= n} \mu_i$ and since $\mu^{\prime} \perp \mu_n$ and $\mu = \mu^{\prime} + \mu_n$, we have that by Theorem 4.2 $\mu^{\prime}$ and $\mu_n$ have the additive property and they are strongly singular. Hence we conclude that each $\mu_n$ has the additive property and they are strongly singular one another.
Next we show that $\limsup_i S_i \cap \cup_{i \ge 1} S_i = \emptyset$. Assume to the contrary that $\limsup_i S_i \cap \cup_{i \ge 1} S_i \not= \emptyset$. Fix some $n \ge 1$ with $\limsup_i S_i \cap S_n \not= \emptyset$ and consider the charge $\mu^{\prime} = \sum_{i \ge 1, i \not=n} \mu_i$. Since the support $S^{\prime}$ of $\mu^{\prime}$ is $\cup_{i \ge 1, i \not= n} S_i \cup \limsup_i S_i$, $S^{\prime} \cap S_n \not= \emptyset$ holds and thus $\mu^{\prime}$ and $\mu_n$ are not strongly singular. But this contradicts Theorem 4.2 by the same arguments above, which completes the proof.
\end{prf}

\section{Application to density measures} 
We consider the asymptotic density on natural numbers $\mathbb{N}$, which is one of the most famous finitely additive set functions on a countable space. Let $\mathcal{P}(\mathbb{N})$ be the set of all subsets of $\mathbb{N}$. For $A \in \mathcal{P}(\mathbb{N})$ the asymptotoic density $d(A)$ of $A$ is defined by
\[d(A) = \lim_n \frac{|A \cap [1, n]|}{n} \]
provided the limit exists, where $|X|$ denotes the number of elements of $X \in \mathcal{P}(\mathbb{N})$. Let $\mathcal{D}$ be the set of all subsets $A \in \mathcal{P}(\mathbb{N})$ having the asymptotic density. Then $d$ is obviously a finitely additive set function defined on $\mathcal{D}$. However, unfortunaltely, the triple $(\mathbb{N}, \mathcal{D}, d)$ is not a charge space since the class $\mathcal{D}$ is not an algebra of sets, i.e., $\mathcal{D}$ does not closed under union or intersection. Though $(\mathbb{N}, \mathcal{D}, d)$ itself is not a charge space, one can construct a charge space from the asymptotic density by extending $d$ to the whole $\mathcal{P}(\mathbb{N})$. A charge $\nu$ defined on $\mathcal{P}(\mathbb{N})$ is called a {\it density measure} if it extends the asymptotic density. An example of density meausres can be constructed simply by using the limit along an ultrafilter in place of the usual limit in the definition of the asymptotic density. 

Recall that an ultrafilter on $\mathbb{N}$ is a filter on $\mathbb{N}$ which is not properly contained in any other filters. In particular, an ultrafilter is said to be free if the intersection of its elements is empty.
Let $\mathcal{U}$ be an ultrafilter on $\mathbb{N}$ and $f$ be in $l_{\infty}$ of the set of all real-valued bounded functions on $\mathbb{N}$. Then there exists a unique number $\alpha$ such that $\{n \in \mathbb{N} : |f(n) - \alpha| < \varepsilon \} \in \mathcal{U}$ holds for any $\varepsilon > 0$. The number $\alpha$ is called the {\it limit of $f$ along $\mathcal{U}$} and denoted by $\mathcal{U}\mathchar`-\lim_n f(n) = \alpha$. Now one can define a density meaure through the limit along an ultrafilter.

Let $\mathcal{U}$ be an free ultrafilter on $\mathbb{N}$, let us define a density measure $\nu_{\mathcal{U}}$ by
\[\nu_{\mathcal{U}}(A) = \mathcal{U}\mathchar`-\lim_n \frac{|A \cap [1,n]|}{n}, \quad A \in \mathcal{P}(\mathbb{N}). \]

Let $\mathcal{C}$ be the set of all such density measures. In what follows, we show a necessary and sufficient condition for a density measure in $\mathcal{C}$ to have the additive property.

Note that there exists distinct ultrafilters $\mathcal{U}$ and $\mathcal{U}^{\prime}$ such that $\nu_{\mathcal{U}} = \nu_{\mathcal{U}^{\prime}}$. Thus, introducing an equivalence relation defined by $\mathcal{U} \sim \mathcal{U}^{\prime}$ if and only if $\nu_{\mathcal{U}} = \nu_{\mathcal{U}^{\prime}}$ on the set of all free ultrafilter on $\mathbb{N}$, $\mathcal{C}$ can be regarded as the quotient space by $\sim$. In fact, there exists a seciton such that each representative has a form which is well suited for investigating the corresponding density measure. Now we have to make some preparation before stating the precise assertion.

Let $\beta\mathbb{N}$ be the  Stone-\v{C}ech compactification of $\mathbb{N}$. Note that $\beta\mathbb{N}$ can be characterized by the following property: for any mapping of $\mathbb{N}$ into a compact space $X$, there exists a continuous extension to $\beta\mathbb{N}$. 
Recall that $\beta\mathbb{N}$ can be identified with the set of all ultrafilters on $\mathbb{N}$ in which a basis of open sets are those subsets $\hat{A} = \{\mathcal{U} : A \in \mathcal{U}\}$ for every $A \in \mathcal{P}(\mathbb{N})$. We denote by $\mathbb{N}^* := \beta\mathbb{N} \setminus \mathbb{N}$ the set of all free ultrafilters on $\mathbb{N}$. Then $\mathbb{N}^*$ is also a compact space by the relative topology of $\beta\mathbb{N}$ and obviously a basis of open sets of $\mathbb{N}^*$ are the sets of the form $A^* = \hat{A} \cap \mathbb{N}^*$ for every $A \in \mathcal{P}(\mathbb{N})$.

We now introduce a topological dynamical system on $\mathbb{N}^*$. Let us consider the translation on $\mathbb{N}$:
\[\tau_0 : \mathbb{N} \rightarrow \mathbb{N}, \quad n \mapsto n+1. \]
Embedding the range $\mathbb{N}$ into $\beta\mathbb{N}$, one can regard $\tau_0$ as a mapping of $\mathbb{N}$ into a compact space $\beta\mathbb{N}$. Thus it can be extended to the continuous function $\tau$ of $\beta\mathbb{N}$ into itself:
\[\tau : \beta\mathbb{N} \rightarrow \beta\mathbb{N}. \]
Though $\tau$ is not a homeomorphism, the restriction of $\tau$ to $\mathbb{N}^*$ is a homeomorphism of $\mathbb{N}^*$ onto itself. We still denote it by the same symbol $\tau$ and then the pair $(\mathbb{N}^*, \tau)$ is a topological dynamical system. Let us define the negative semi-orbit of $\mathcal{U} \in \mathbb{N}^*$ by $o_-(\mathcal{U}) := \{\tau^{-n}\mathcal{U} : n = 0, 1, 2, \ldots. \}$. $o_-(\mathcal{U})$ is said to be recurrent if for any neighborhood $U$ of $\mathcal{U}$ and any natural number $N \ge 1$, there exists some $n \ge N$ such that $\tau^{-n}\mathcal{U} \in U$ holds. The set of all such points in $\mathbb{N}^*$ is denoted by $\mathcal{R}_{d, -}$.

Next we extend this discrete flow to a continuous flow which is the suspension of $(\mathbb{N}^*, \tau)$. Let us consider the product space $\mathbb{N}^* \times [0, 1]$ and define an equivalence relation on it by $(\mathcal{U}, 1) \sim (\tau\mathcal{U}, 0)$ for every $\mathcal{U} \in \mathbb{N}^*$. Let $\Omega^*$ be the quotient space of $\mathbb{N}^* \times [0, 1]$ by $\sim$. Then we define a family of homeomorphisms $\tau^s : \Omega^* \rightarrow \Omega^*$ by 
\[\tau^s(\mathcal{U}, t) = (\tau^{[t+s]}\mathcal{U}, t+s-[t+s]), \]
for each $s \in \mathbb{R}$, where $[x]$ denotes the largest integer not exceeding $x$ for a real number $x$. Then it is readily verified that the pair $(\Omega^*, \{\tau^s\}_{s \in \mathbb{R}})$ is a continuous flow. Similarly as above, we define the negative semi-orbit of $\omega \in \Omega^*$ by $o_-(\omega) := \{\tau^{-s}\omega : s \ge 0 \}$. Also, we say that $o_-(\omega)$ is recurrent if for any neighborhood $U$ of $\omega$ and positive real number $R > 0$, there exists some $s \ge R$ such that $\tau^{-s} \omega \in U$ holds. The set of all recurrent points in $\Omega^*$ is denoted by $\mathcal{R}_-$.

Now the following theorem asserts that each element of $\mathcal{C}$ is expressed in a special form, which is convenient to investigate its measure theoretic properties. See [4], [5] for the proof and related results.

\begin{thm}
Let us define the mapping $\Phi: \Omega^* \rightarrow \mathcal{C}$ as follows: Let $\omega = (\mathcal{U}, t) \in \Omega^*$ and $\theta = 2^t$, and let us denote the image $\Phi(\omega)$ of $\omega$ by $\nu_{\omega}$, then we define
\[\nu_{\omega}(A) = \mathcal{U}\mathchar`-\lim_n \frac{|A \cap [1, [\theta \cdot 2^n]]|}{\theta \cdot 2^n}, \quad A \in \mathcal{P}(\mathbb{N}). \]
Then $\Phi$ is a one to one and onto mapping of $\Omega^*$ to $\mathcal{C}$.
\end{thm}

Before proving the main theorem, we introduce the following auxiliary charges; for $\omega = (\mathcal{U}, t) \in \Omega^*$ and $m = 0, 1, \ldots$, we define a charge $\nu_{\omega, n} : \mathcal{P}(\mathbb{N}) \rightarrow [0, 1]$ by
\[\nu_{{\omega}, m}(A) = \mathcal{U}\mathchar`-\lim_n \frac{|A \cap ([\theta \cdot 2^{n-m-1}], [\theta \cdot 2^{n-m}]]|}{\theta \cdot 2^{n-m-1}}, \quad A \in \mathcal{P}(\mathbb{N}). \]
Then it is obvious that
\[\nu_{\omega} = \sum_{m=0}^{\infty} \frac{1}{2^{m+1}} \nu_{{\omega}, m}. \]

\begin{lem}
For any $\omega = (\mathcal{U}, t) \in \Omega^*$ and $m = 0, 1, 2, \ldots$, $\nu_{\omega, m}$ has the additive property. 
\end{lem}

\begin{prf}
Given an increasing sequence $A_1 \subseteq A_2 \subseteq \ldots \subseteq A_i \subseteq \ldots$ of $\mathcal{P}(\mathbb{N})$. Set $\lim_{i \to \infty} \nu_{{\omega}, m}(A_i) = \alpha$. We take a decreasing sequence $\{X_i\}_{i \ge1}$ of $\mathcal{U}$ such that
\[\left|\frac{|A_i \cap ([\theta \cdot 2^{n-m-1}], [\theta \cdot 2^{n-m}]]|}{\theta \cdot 2^{n-m-1}} - \nu_{\omega, m}(A_i) \right| < \frac{1}{i}  \]
whenever $n \in X_i$. Then we define $B \subseteq \mathbb{N}$ as $B \cap ([\theta \cdot 2^{n-m-1}], [\theta \cdot 2^{n-m}]] = A_i \cap ([\theta \cdot 2^{n-m-1}], [\theta \cdot 2^{n-m}]]$ if $n \in X_i \setminus X_{i+1}$ and $B \cap ([\theta \cdot 2^{n-m-1}], [\theta \cdot 2^{n-m}]] = \emptyset$ otherwise.
First we show that $\nu_{{\omega}, m}(B) = \alpha$. For any $\varepsilon > 0$, take $i \in \mathbb{N}$ with $\varepsilon > \frac{1}{i}$ and $\alpha - \nu_{\omega, m}(A_i) < \varepsilon$. Then for any $n \in X_i$, there exists some $j_n \ge i$ such that $n \in X_{j_n} \setminus X_{j_n+1}$. Hence
\begin{align}
\left|\nu_{{\omega}, m}(B) - \alpha \right| &= \mathcal{U}\mathchar`-\lim_n \left|\frac{|B \cap ([\theta \cdot 2^{n-m-1}], [\theta \cdot 2^{n-m}]]|}{\theta \cdot 2^{n-m-1}} - \alpha \right| \notag \\
&\le \limsup_{n \in X_i} \left|\frac{|B \cap ([\theta \cdot 2^{n-m-1}], [\theta \cdot 2^{n-m}]]|}{\theta \cdot 2^{n-m-1}} - \nu_{\omega, m}(A_{j_n}) \right| + \left|\nu_{\omega, m}(A_i) - \alpha \right| \notag \\
&= \limsup_{n \in X_i} \left|\frac{|A_{j_n} \cap ([\theta \cdot 2^{n-m-1}], [\theta \cdot 2^{n-m}]]|}{\theta \cdot 2^{n-m-1}} - \nu_{\omega, m}(A_{j_n}) \right| + \left|\nu_{\omega, m}(A_i) - \alpha \right| \notag \\
&\le \frac{1}{j_n} + \varepsilon \le \frac{1}{i} + \varepsilon < 2\varepsilon. \notag
\end{align}
Since $\varepsilon > 0$ is arbitrary, we have $\nu_{{\omega}, m}(B) = \alpha$. Next we show that $\nu_{{\omega}, m}(A_i \setminus B) = 0$ for every $i \ge 1$. For any $n \in X_i$, there exists some $j_n \ge i$ such that $n \in X_{j_n} \setminus X_{{j_n}+1}$. Hence we have
\[\nu_{{\omega}, m}(A_i \setminus B) \le \limsup_{n \in X_i} \left|\frac{|(A_i \setminus B) \cap ([\theta \cdot 2^{n-m-1}], [\theta \cdot 2^{n-m}]]|}{\theta \cdot 2^{n-m-1}} \right| = 0 \]
since $B \cap ([\theta \cdot 2^{n-m-1}], [\theta \cdot 2^{n-m}]] = A_{j_n} \cap ([\theta \cdot 2^{n-m-1}], [\theta \cdot 2^{n-m}]]$ and $A_i \subseteq A_{j_n}$. So we obtain the result.

\end{prf}

Based on the above lemma, now we prove the main result:

\begin{thm}
$\nu_{\omega}$ has the additive property if and only if $\omega \not\in \mathcal{R}_-$.
\end{thm}

\begin{proof}
Let $S_m$ = supp $\nu_{{\omega}, m}$ and $S$ = supp $\nu_{\omega}$. Then by Theorem 4.3 and the fact mentioned above it is equivalent to show that 
\[\left(\limsup_m S_m \right) \bigcap \bigcup_{i \ge 1} S_i = \emptyset \ \text{if and only if} \ \omega \not\in \mathcal{R}_-. \] 

Assume that $\omega = (\mathcal{U}, t) \not\in \mathcal{R}_-$ and thus $\mathcal{U} \not\in \mathcal{R}_{d, -}$. This means that there exists some $X \in \mathcal{U}$ such that $X^* \cap o_-(\mathcal{U}) \setminus \{\mathcal{U}\} = \emptyset$. Put
\[I_0 = \bigcup_{n \in X} ([\theta \cdot 2^{n-1}], [\theta \cdot 2^n]] \]
and it is obvious that $\nu_{{\omega}, 0}(I_0) = 1$, which means that $S_0 \subseteq I_0^*$, that is, $I_0^*$ is a neighborhood of $S_0$. Now we show that $I_0^* \cap S_m = \emptyset$ for every $m \ge 1$. In fact, take $X_m \in \mathcal{U}$ such that $\tau^{-m} X_m \cap X = \emptyset$ and put 
\[I_m = \bigcup_{n \in \tau^{-m}X_m}([\theta \cdot 2^{n-1}], [\theta \cdot 2^n]]. \]
Then we have $\nu_{{\omega}, m}(I_m) =1$ and thus $S_m \subseteq I_m^*$. Since we have assumed that $\tau^{-m} X_m \cap X = \emptyset$, we have $I_0^* \cap I_m^* = \emptyset$. Hence for the neighborhood $I_0^*$ of $S_0$ we have that $I_0^* \cap S_m= \emptyset$ for all $m \ge 1$ and thus $\limsup_m S_m \cap S_0 = \emptyset$. In the same way, we can show that $\limsup_m S_m \cap S_i = \emptyset$ for every $i \ge 1$ and thus we have $\limsup_m S_m \cap \cup_{i \ge 1} S_i = \emptyset$.

\medskip
On the other hand, assume that $\omega \in \mathcal{R}_-$, i.e., $\mathcal{U}$ is in the closure of $o_-(\mathcal{U})$. Then notice that for any $X \in \mathcal{U}$ and positive integer $N > 0$ there exists some $n \ge N$ such that $\tau^{-n} \mathcal{U} \in X^*$. We show that for any neighborhood $I_0^*$ of $S_0$ where $I_0 \in \mathcal{P}(\mathbb{N})$ and positive integer $N$, there is some $n \ge N$ such that $I_0^* \cap S_n \not= \emptyset$, which immediately implies that $\limsup_m S_m \cap \cup_{i \ge 0} S_i \not= \emptyset$. Let $f \in l_{\infty}$ be the function defined by
\[f(m) = |I_0 \cap ([\theta \cdot 2^{m-1}], [\theta \cdot 2^m]]|/\theta \cdot 2^{m-1}, \quad m = 1,2, \ldots. \] 
Then $\nu_{{\omega}, 0}(I_0) = \overline{f}(\mathcal{U})$ holds by the definition of $\nu_{\omega, 0}$. Since $\overline{f}(\mathcal{U}) = \nu_{{\omega}, 0}(\mathbb{N}) =1$ and $\overline{f}$ is continuous on $\mathbb{N}^*$, there exists a neighborhood $X^*$ of $\mathcal{U}$ such that $\mathcal{U}^{\prime} \in X^*$ implies $\overline{f}(\mathcal{U}^{\prime}) > 0$. Let $n \ge N$ be any integer such that $\tau^{-n}\mathcal{U} \in X$. Then $\nu_{{\omega}, n}(I_0) = 2^n \cdot \overline{f}(\tau^{-n}\mathcal{U}) > 0$, which means that $S_n \cap I_0^* \not= \emptyset$. We completes the proof.
\end{proof}


\bigskip

\begin{thebibliography} {15}
\bigskip
\bigskip

\bibitem {book1} A. Basile, K. P. S. Bhaskara Rao, {\em Completeness of $L_p$-Spaces in the Finitely 
Additive Setting and Related Stories,} J. Math. Anal. Appl. \textbf{248}, (2000) 588-624.
\bibitem {book3} K.P.S. Bhaskara Rao, M. Bhaskara Rao, {\em Theory of charges,} Academic Press (1983). 
\bibitem {book4} A. Blass, R. Frankiewicz, G. Plebanek, C. Ryll-Nardzewski, {\em A note on extensions of asymptotic density,} Proc. Amer. Math. Soc. \textbf{129} (2001) 3313-3320.
\bibitem {book4} R. Kunisada, {\em Density measures and additive property}, J. Number Theory. \textbf{176} (2017), 184-203.
\bibitem {book5} R. Kunisada, {\em On a relation between density measures and a certain flow}, Proc. Amer. Math. Soc. \textbf{147} (2019) 1941-1951.

\end{thebibliography}
\end{document}